\title{Polychromatic Coloring for Half-Planes.}

\author{Shakhar Smorodinsky \thanks{Ben-Gurion University,
Be'er Sheva 84105, Israel. {\tt shakhar@math.bgu.ac.il}} \and Yelena
Yuditsky \thanks{Ben-Gurion University, Be'er Sheva 84105,
Israel. {\tt yuditsky@bgu.ac.il}. This work was done while the
2nd author was an M.Sc. student at Ben-Gurion University under
the supervision of Shakhar Smorodinsky.}}

\documentclass[11pt]{article}

\usepackage{a4wide,color}
\usepackage{cite}
\usepackage{amsfonts}
\usepackage{graphicx,amssymb,amsmath,amsthm,subfigure}
\usepackage[ruled,vlined]{algorithm2e}

\renewcommand{\Re}{\mathbb{R}}
\newcommand{\R}{\cal{R}}
\renewcommand{\H}{\cal{H}}
\newcommand{\E}{\cal{E}}

\definecolor{SHAKHAR}{rgb}{0,0.5,0.5}

\newtheorem{theorem}{Theorem}[section]
\newtheorem{lemma}[theorem]{Lemma}

\newtheorem{proposition}[theorem]{Proposition}

\newcommand{\tr}{{\mathcal T}} 

\newcommand{\VC}{\ensuremath{\mathsf{VC}}\xspace}
\newcommand{\Komlos}{K\'o{}m{}l{}o{}s\xspace}

\newcommand{\eps}{\epsilon}
\newcommand{\pth}[2][\!]{#1\left({#2}\right)}

\begin{document}
\maketitle
\begin{abstract}
We prove that for every integer $k$, every finite set of points in
the plane can be $k$-colored so that every half-plane that
contains at least $2k-1$ points, also contains at least one point
from every color class. We also show that the bound $2k-1$ is
best possible. This improves the best previously known lower and
upper bounds of $\frac{4}{3}k$ and $4k-1$ respectively. We also show that every finite set of half-planes can be $k$ colored so that if a point $p$ belongs to a subset $H_p$ of at least $3k-2$ of the half-planes then $H_p$ contains a
half-plane from every color class. This improves the best
previously known upper bound of $8k-3$. Another corollary of our
first result is a new proof of the existence of small size
$\eps$-nets for points in the plane with respect to half-planes.
\end{abstract}

\section{Introduction}

In this contribution, we are interested in coloring finite sets
of points in $\Re^2$ so that any half-plane that contains at
least some fixed number of points, also contains at least one
point from each of the color classes.

Before stating our results, we introduce the following definitions:

A {\em range space} (or {\em hypergraph}) is a pair $(V,\E)$ where
$V$ is a set (called the {\em ground set}) and $\E$ is a set of
subsets of $V$.

A {\em coloring} of a hypergraph is an assignment of colors to the
elements of the ground set. A {\em $k$-coloring} is a
function $\chi:V \rightarrow \{1,\ldots,k\}$. A hyperedge $S \in
\E$ is said to be {\em polychromatic} with respect to some
$k$-coloring $\chi$ if it contains a point from each of the $k$
color classes. That is, for every $i \in \{1,\ldots,k\}$ $S \cap
\chi^{-1}(i) \neq \emptyset$. We are interested in hypergraphs
induced by an infinite family of geometric regions. Let $\R$ be a
family of regions in $\Re^d$ (such as all balls, all
axis-parallel boxes, all half-spaces, etc.)

Consider the following two functions defined for $\R$ (notations
are taken from \cite{ACCLS09}):
\begin{enumerate}
\item
Let $f=f_{\R}(k)$ denote the minimum number such that any finite
point set $P \subset \Re^d$ can be $k$-colored so that every
range $R \in \R$ containing at least $f$ points of $P$ is
polychromatic.

\item Let $\bar{f}= \bar{f}_{\R}(k)$ denote the minimum number such that
any finite sub-family $\R' \subset \R$ can be
$k$-colored so that for every point $p \in \Re^d$, for which the
subset ${\R}'_p \subset \R'$ of regions containing $p$ is of size
at least $\bar{f}$, ${\R}'_p$ is polychromatic.
\end{enumerate}

We note that the functions $f_{\R}(k)$ and $\bar{f}_{\R}(k)$ might not be bounded even for $k=2$. Indeed, suppose $\R$ is the family of all convex
sets in the plane and $P$ is a set of more than $2f-2$ points in
convex position. Note that any subset of $P$ can be cut-off by some range in $\R$. By the pigeon-hole principle, any $2$ coloring
of $P$ contains a monochromatic subset of at least $f$ points,
thus illustrating that $f_{\R}(2)$ is not bounded in that case.
Also note that $f_{\R}(k)$ and $\bar{f}_{\R}(k)$ are monotone
non-decreasing, since any upper bound for $f_{\R}(k)$ would imply
an upper-bound for $f_{\R}(k-1)$ by merging color classes. We
sometimes abuse the notation and write $f(k)$ when the family of
ranges under consideration is clear from the context.

The functions defined above are related to the so-called {\em
cover-decomposable} problems or the decomposition of
\emph{$c$-fold coverings} in the plane. It is a major open
problem to classify, for which families $\R$ those functions are
bounded, and in those cases to provide sharp bounds on $f_{\R}(k)$
and $\bar{f}_{\R}(k)$. Pach~\cite{Pa80} conjectured that
$\bar{f}_{\tr}(2)$ exists whenever $\tr$ is a family of all
translates of some fixed compact convex set. These functions have
been the focus of many recent research papers and some special
cases are resolved. See, e.g.,
\cite{strips-latin,ACCLOR09,GV09,MP86,PTT07,PT07,Pal09,PT10,TT07}.
We refer the reader to the introduction of \cite{ACCLS09} for
more details on this and related problems.

\paragraph{Application to Battery Consumption in Sensor Networks} Let $\R$ be a
collection of sensors, each of which monitors the area within a
given shape $A$. Assume further that each sensor has a battery
life of one time unit. The goal is to monitor the region $A$ for
as long as possible. If we activate all sensors in $\R$
simultaneously, $A$ will be monitored for only one time unit.
This can be improved if $\R$ can be partitioned into $c$ pairwise
disjoint subsets, each of which covers $A$. Each subset can be
used in turn, allowing us to monitor $A$ for $c$ units of time.
Obviously if there is a point in $A$ covered by only $c$ sensors
then we cannot partition $\R$ into more than $c$ families.
Therefore, it makes sense to ask the following question: what is
the minimum number $\bar{f}(k)$ for which we know that, if every point
in $A$ is covered by $\bar{f}(k)$ sensors, then we can partition $\R$
into $k$ pairwise disjoint covering subsets? This is exactly the
type of problem that we described. For more on the relation
between these partitioning problems and sensor networks, see the
paper of Buchsbaum~\emph{et al.}~\cite{othersensors}.

\paragraph{Our results}
 For the family $\H$ of all half-planes, Pach and T\'{o}th
showed in \cite{PT07} that $f_{\H}(k)=O(k^2)$. Aloupis~\emph{et al.}~\cite{ACCLS09} showed that $\frac{4k}{3} \leq f_{\H}(k)\leq 4k-1
$. In this paper, we settle the case of half-planes by showing
that the exact value of $f_{\H}(k)$ is $2k-1$. Keszegh \cite{K07}
showed that $\bar{f}_{\H}(2) \leq 4$ and Fulek~\cite{Fulek}
showed that $\bar{f}_{\H}(2)= 3$.  Aloupis~\emph{et al.}~\cite{ACCLS09}
showed that $\bar{f}_{\H}(k) \leq 8k-3$. In this paper, we obtain the improved bound of $\bar{f}_{\H}(k) \leq 3k-2$.

\paragraph{An Application to $\eps$-Nets for Half-Planes}
Let $H=(V,\E)$ be a hypergraph where $V$ is a finite set. Let
$\epsilon \in [0,1]$ be a real number. A subset $N\subseteq V$ is
called an \textit{$\epsilon$-net} if for every hyperedge $S\in
\E$ such that $|S|\geq \epsilon |V|$, we have also $S \cap N \neq
\emptyset$. In other words, $N$ is a hitting set for all ``large"
hyperedges. Haussler and Welzl \cite{HW} proved the following
fundamental theorem regarding the existence of small $\eps$-nets
for hypergraphs with a small \VC-dimension.

\begin{theorem}[$\eps$-net theorem \cite{HW}] 
    Let $H=(V,\E)$ be a hypergraph with \VC-dimension $d$. For
    every $\eps \in (0,1]$, there
    exists an $\eps$-net $N \subset V$ with cardinality at most
    $\displaystyle O\pth{  \frac{d}{\eps}\log\frac{1}{\eps} }$
\end{theorem}

The notion of $\epsilon$-nets is central in several mathematical
fields, such as computational learning theory, computational
geometry, discrete geometry and discrepancy theory.

Most hypergraphs studied in discrete and computational geometry
have a finite \VC-dimension. Thus, by the above-mentioned
theorem, these hypergraphs admit small size $\eps$-nets. \Komlos
~\emph{et al.}~\cite{KPW} proved that the bound
$O(\frac{d}{\eps}\log\frac{1}{\eps})$ on the size of an $\eps$-net
for hypergraphs with \VC-dimension $d$ is best possible. Namely,
for a constant $d$ they construct a hypergraph $H$ with
\VC-dimension $d$ such that any $\eps$-net for $H$ must have
size of at least $\Omega(\frac{1}{\eps}\log\frac{1}{\eps})$.
However, their construction is random and seems far from being a
``nice'' geometric hypergraph. It is believed that for most
hypergraphs with \VC-dimension $d$ that arise in the geometric
context, one can improve on the bound
$O(\frac{d}{\eps}\log\frac{1}{\eps})$.

Consider a hypergraph $H=(P,E)$ where $P$ is a finite set of points
in the plane and
$$E=\{P\cap h: \mbox{h is a half-plane}\}.$$ For
this special case,  Woeginger \cite{GW} showed that for any
$\epsilon
> 0$ there exists an $\epsilon$-net for $H$ of size at most
$\frac{2}{\epsilon}-1$ (see also, \cite{PW90}).

As a corollary of our result, we obtain yet another proof for this fact.


\section{Coloring Points with Respect to Half-Planes}

Let $\H$ denote the family of all half-planes in $\Re^2$. In this
section we prove our main result by finding the exact value of
$f_{\H}(k)$, for the family $\H$ of all half-planes.

\begin{theorem} \label{thm-main}
$f_{\H}(k) = 2k-1$.
\end{theorem}

We start by proving the lower bound $f_{\H}(k)\geq 2k-1$. Our lower bound construction is simple, and is inspired by a lower bound construction for $\eps$-nets with respect to half-planes given in \cite{GW}. We need
to show that there exists a finite set $P$ in $\Re^2$ such that
for every $k$-coloring of $P$ there is a half-plane that contains
$2k-2$ points, and is not polychromatic. In fact, we show a
stronger construction. For every $n \geq 2k-1$ there is such a
set $P$ with $|P|=n$. We construct $P$ as follows: We place
$2k-1$ points on a concave curve $\gamma$ (e.g., the parabola
$y=x^2$, $-1<x<1$). Let $p_1,p_2,..,p_{2k-1}$ be the points
ordered from left to right along their $x$-coordinates. Notice
that for every point $p_i$ on $\gamma$ there is an open positive
half-plane $h_i$ that does not contain $p_i$ and contains the
rest of the $2k-2$ points that are on $\gamma$. Namely, $h_i \cap
\{p_1,...,p_{2k-1}\} = \{p_1,...p_{i-1},p_{i+1},...,p_{2k-1}\}$.
We choose $h_1,h_2,...,h_{2k-1}$ in such a way that
$\cap_{i=1}^{2k-1}\overline{h_i}\neq \emptyset$ where
$\overline{h_i}$ is the complement of $h_i$. We place $n-(2k-1)$
points in $\cap_{i=1}^{2k-1}\overline{h_i}$. Let
$\chi:P\rightarrow \{1,...,k\}$ be some $k$-coloring of $P$.
There exists a color $c$ that appears at most once among the
points on $\gamma$ (for otherwise we would have at least $2k$
points). If no point on $\gamma$ is colored with $c$ then a
(positive) half-plane bounded by a line separating the parabola
from the rest of the points is not polychromatic. Let $p_j$ be
the point colored with $c$. As mentioned, the open half-plane
$h_j$ contains all the other points on $\gamma$ (and only them),
so $h_j$ contains $2k-2$ points and misses the color $c$. Hence,
it is not polychromatic. Thus $f_{\H}(k)
> 2k-2$ and this completes the lower bound construction. See
Figure \ref{fig-1} for an illustration.

\begin{figure}
    \centering
    \includegraphics[width=1in]{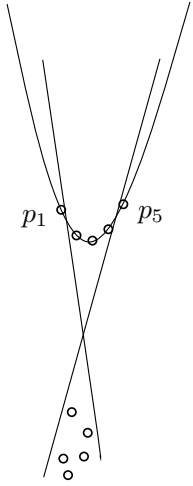}
    \caption{\label{fig-1}A construction showing that $f(k)>2k-2$ for $n=10$ and $k=3$}
\end{figure}

Next, we prove the upper-bound $f_{\H}(k) \leq 2k-1$. In what
follows, we assume without loss of generality that the set of
points $P$ under consideration is in general position, namely,
that no three points of $P$ lie on a common line. Indeed, we can
slightly perturb the point set to obtain a set $P'$ of points in
general position. The perturbation is done in such a way that for
any subset of the points of the form $h \cap P$ where $h$ is a
half-plane, there is another half-plane $h'$ such that $h \cap P
= h'\cap P'$. Thus any valid polychromatic $k$-coloring for $P'$
also serves as a valid polychromatic $k$-coloring for $P$.

For the proof of the upper bound we need the following lemma:

\begin{lemma}
Let $P$ be a finite point set in the plane in general position and
let $t\geq 3$ be some fixed integer. Let $H'=(P,\cal {E'})$ be a
hypergraph where ${\cal E'}=\{ P \cap h: h \in {\H}, |P\cap h|=t
\}$. Let $P'\subseteq P$ be the set of extreme points of $P$
(i.e., the subset of points in $P$ that lie on the boundary of the
convex-hull $CH(P)$ of $P$). Let $N\subseteq P'$ be a
(containment) minimal hitting set for $H'$. Then for every $E\in
\cal {E'}$ we have $|N\cap E|\leq 2$. \label{lemma1}
\end{lemma}

\begin {proof}

First notice that such a hitting set $N \subset P'$ for $H'$
indeed exists since $P'$ is a hitting set.

Assume to the contrary that there exists a hyperedge $E\in \cal
{E'}$ such that $|N\cap E|\geq 3$. Let $h$ be a half-plane such
that $h\cap P=E$ and let $l$ be the line bounding $h$. Assume,
without loss of generality, that $l$ is parallel to the $x$-axis
and that the points of $E$ are below $l$. If $l$ does not
intersect the convex hull $CH(P)$ or is tangent to $CH(P)$ then
$h$ contains $P$ and $|P|=t$. Thus any minimal hitting set $N$
contains exactly one point of $P$, a contradiction. Hence, the
line $l$ must intersect the boundary of $CH(P)$ in two points.

Let $q,q'$ be the left and right points of  $l\cap
\partial CH(P)$ respectively.
Let $p,r,u$ be three points in $N\cap E$ ordered according to
their counter-clockwise order on $\partial CH(P)$. By the
minimality property, there is a half-plane $h_r$ such that $h_r
\cap P \in {\cal E}' $ and such that $N\cap h_r=\{r\}$, for
otherwise, $N\setminus\{ r \}$ is also a hitting-set for $\cal
{E'}$ contradicting the minimality of $N$. See Figure
\ref{lemma_fig} for an illustration.

Denote the line bounding $h_r$ by $l_r$ and denote by $\bar{h}_r$ the complement half-plane of $h_r$. Notice that $l_r$ can
not intersect the line $l$ in the interior of the segment $qq'$.
Indeed assume to the contrary that $l_r$ intersects the segment
$qq'$ in some point $x$. Then, by convexity, the open segment $rx$
lies in $h_r$. However, the segment $rx$ must intersect the segment $pu$. This is impossible since both $p$ and $u$ lie in $\bar{h}_r$ and therefore, by convexity also the segment $pu$ lies in $\bar{h}_r$. Thus the segment $pu$
and the segment $rx$ are disjoint.

Next, suppose without loss of generality that the line $l_r$
intersects $l$ to the right of the segment $qq'$. Let $q''$
denote the point $l \cap l_r$. We have that $|h_r\cap P|=t$ and
also $|E|=|h\cap P| = t$, therefore there is at least one point $r'$ that
is contained in $h_r \cap P$ and is not contained in $h$, hence
it lies above the line $l$. The segment $rr'$ must intersect the
line $l$ to the right of the point $q''$. Also, by convexity, the
segment $rr'$ is contained in $CH(P)$. This implies that the line
$l$ must intersect $\partial CH(P)$ to the right of $q'$, i.e
intersects $\partial CH(P)$ in three points, a contradiction.

\end {proof}

\begin{figure}
    \centering
    \includegraphics[width=0.7\textwidth]{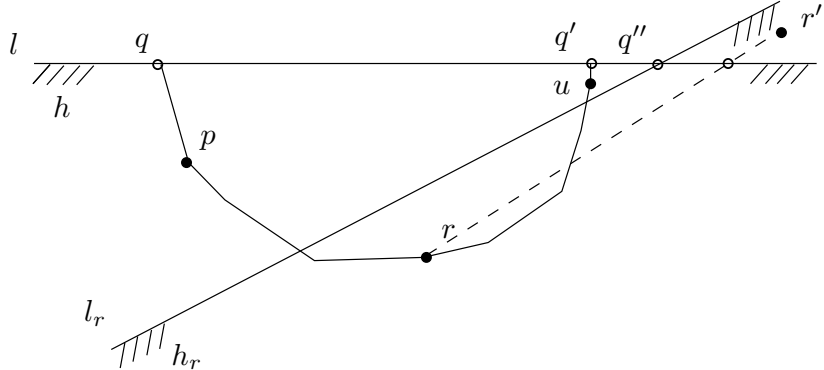}
    \caption{ \label{lemma_fig}
The line $l$ intersects the boundary of $CH(P)$ in two points.}
\end{figure}

We are ready to prove the second part of Theorem~\ref{thm-main}:
Recall that for a given finite planar set $P\subset \Re^2$ and an
integer $k$, we need to show that there is a $k$-coloring for $P$
such that every half-plane that contains at least $2k-1$ points
is polychromatic.

For $k=1$ the theorem is obvious. For $k=2$, put $t=3$ and let $N$
be a hitting set as in lemma \ref{lemma1}. We assign the points
of $N$ the color $2$ and assign the points of $P\setminus N$ the
color $1$. Let $h$ be a half-plane such that $|h\cap P|\geq 3$.
Assume without loss of generality that $h$ is a negative
half-plane. Let $l$ denote the line bounding $h$. Translate $l$
downwards to obtain a line $l'$, such that for the negative
half-plane $h'$ bounded by $l'$, we have $h'\cap P\subseteq h\cap
P$ and $|h'\cap P|=3$. We can assume without loss of generality
that no line parallel to $l$ passes through two points of $P$.
Indeed, this can be achieved by rotating $l$ slightly. Obviously
$h' \cap N \neq \emptyset$. Moreover, by lemma \ref{lemma1} we
have that $h'\cap (P\setminus N) \neq \emptyset$. Hence, $h'$
contains both a point colored with $1$ and a point colored with
$2$, i.e., $h'$ is polychromatic. Thus $h$ is also polychromatic.

We prove the theorem by induction on the number of colors $k$.
The induction hypothesis is that the theorem holds for all values
$i < k$. Let $k > 2$ be an integer. Put $t=2k-1$ and let $N$ be a
minimal hitting set as in Lemma~\ref{lemma1}. We assign all
points in $N$ the color $k$. Put $P' = P \setminus N$. By the
induction hypothesis, we can color the points of $P'$ with $k-1$
colors, such that for every half-plane $h$ with $|h \cap P'| \geq
2k-3$, $h$ is polychromatic, i.e., $h$ contains representative
points from all the $k-1$ color classes. We claim that this
coloring together with the color class $N$ forms a valid
$k$-coloring for $P$. Consider a half-plane $h$ such that $|h\cap
P|\geq 2k-1$. As before, let $h'$ be a half-plane such that
$h'\cap P\subseteq h\cap P$ and $|h'\cap P|=2k-1$. It is enough
to show that $h'$ is polychromatic. By lemma \ref{lemma1} we know
that $1\leq|h'\cap N|\leq 2$, therefore we can find a half-plane
$h''$ such that $h''\cap P\subseteq h'\cap P$ and $|h''\cap
(P\setminus N)|=2k-3$. By the induction hypothesis, $h''$ contain
representative points from all the initial $k-1$ colors. Thus $h'$
contain a point from $N$ (i.e., colored with $k$) and a point
from each of the initial $k-1$ colors. Hence $h'$ is polychromatic
and so is $h$. This completes the proof of the theorem.

{\bf \noindent Remark:} The above theorem also provides a
recursive algorithm to obtain a valid $k$-coloring for a given
finite set $P$ of points. See Algorithm \ref{alg:mine}. Here, we do not care about the running time of the algorithm. Assume that we have a ``black-box'' that
finds a hitting set $N$ as in lemma~\ref{lemma1} in time bounded
by some function $f(n,t)$.

\begin{algorithm}

\caption{Algorithm for polychromatic $k$-coloring}

\SetLine \KwIn{A finite set $P\subset \Re^2$ and an integer
$k\geq1$} \KwOut{A polychromatic $k$-coloring $\chi:P\rightarrow
\{1,...,k\}$}
\Begin{
 \If {k=1} {Color all points of $P$ with
color 1.} \Else
{Find a minimal hitting set $N$ as in lemma \ref{lemma1} for all the half-planes of size $2k-1$.\\
Color the points in $N$ with color $k$.\\
 Set $P=P\setminus N$ and $k=k-1$. Recursively color $P$ with $k$ colors. }

\label{alg:mine}
}
\end{algorithm}

Note that a trivial bound on the total running time of the
algorithm is $\sum_{i=1}^k f(n,2i-1)$.


\section{Coloring Half-Planes with Respect to Points}
Keszegh \cite{K07} investigated the value of $\bar{f}_{\H}(2)$ and
proved that $\bar{f}_{\H}(2)\leq 4$. Recently Fulek \cite{Fulek}
showed that in fact $\bar{f}_{\H}(2) = 3$. For the general case,
Aloupis~\emph{et al.} proved in \cite{ACCLS09} that $\bar{f}_{\H}(k)\leq
8k-3$. We obtain an improved bound of $\bar{f}_{\H}(k)\leq 3k-2$. Before proving this bound, let us show a simple proof of the weaker bound $\bar{f}_{\H}(k)\leq 4k-3$. 

\begin{theorem}
\label{dual}
$\bar{f}_{\H}(k) \leq 4k-3$
\end{theorem}

Theorem~\ref{dual} is a direct corollary of
Theorem~\ref{thm-main} and uses a reduction to coloring points in the plane. This reduction was also used in \cite{ACCLS09}.  

\begin{proof}
Let $H\subseteq {\H}$ be a finite set of half-planes. We partition
$H$ into two disjoint sets $H^+$ and $H^-$ where $H^+ \subset H$
(respectively $H^- \subset H$) is the set of all positive
half-planes (respectively negative half-planes). It is no loss of
generality to assume that all lines bounding the half-planes in
$H$ are distinct. Indeed, by a slight perturbation of the lines,
one can only obtain a superset of hyperedges in the corresponding
hypergraph (i.e, a superset of cells in the arrangement of the
bounding lines). Let $L^+$ (respectively $L^-$) be the sets of
lines bounding the half-planes in $H^+$ (respectively $H^-$).
Next, we use a standard (incidence-preserving) dualization to
transform the set of lines $L^+$ (respectively $L^-$) to a set of
points ${L^+}^*$ (respectively ${L^-}^*$). It has the property
that a point $p$ is above (respectively incident or below) a line
$l$ if and only if the dual line $p^*$ is above (respectively
incident or below) the point $l^*$. See Figure \ref{red-blue} for
an illustration.

\begin{figure}
    \centering
    \includegraphics[width=0.45\textwidth]{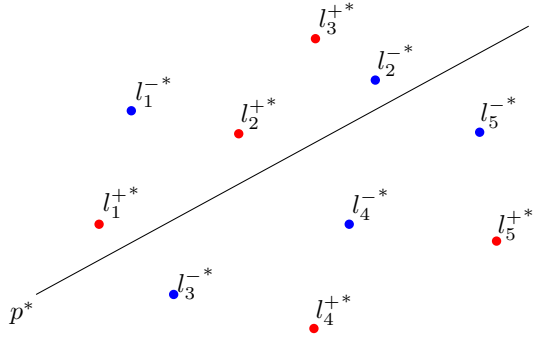}
    \caption{   \label{red-blue}
An illustration of the dualization. In the primal, the point $p$
is contained in the half-planes
    bounded by the lines $l_1^-,l_2^-,l_4^+$ and $l_5^+$}
\end{figure}

We then color the sets ${L^+}^*$ and ${L^-}^*$
independently. We color each of them with $k$-colors so that
every half-plane containing $2k-1$ points of a given set, is also
polychromatic. Obviously, by Theorem~\ref{thm-main}, such a
coloring can be found. This coloring induces the final coloring
for the set $H = H^+ \cup H^-$. To prove that this coloring is
indeed valid, consider a point $p$ in the plane. Let $H'\subseteq
H$ be the set of half-planes containing $p$. We claim that if
$|H'|\geq 4k-3$ then $H'$ is polychromatic. Indeed, if $|H'|\geq
4k-3$ then, by the pigeon-hole principle, either $|H'\cap
H^+|\geq 2k-1$ or $|H'\cap H^-|\geq 2k-1$. Suppose without loss
of generality that $|H'\cap H^+|\geq 2k-1$. Let
$L^+_{H'}\subseteq L^+$ be the set of lines bounding the
half-planes in $H'\cap H^+$. In the dual, the points in
${L^+_{H'}}^*$ are in the half-plane below the line $p^*$. Since
$|{L^+_{H'}}^*|\geq 2k-1$, we also have that ${L^+_{H'}}^*$ is
polychromatic, thus the set of half-planes $H'\cap H^+$ is
polychromatic and so is the set $H'$. This completes the proof of
the theorem.
\end {proof}

As promised, we further improve the bound on $\bar{f}_{\H}(k)$.

\begin{theorem}
\label{dual-better}
$\bar{f}_{\H}(k) \leq 3k-2$
\end{theorem}

In the proof of Theorem \ref{dual-better} we use polar point-line duality (which was also used in \cite{Fulek}). However, this duality applies only if there exists a point $p\in \Re^2$ that is not covered by any half-plane in $H$. If this is not the case, we use the following known fact.

\begin{proposition}
\label{fact-halfplanes}
Let $H$ be a set of half-spaces in $\Re^d$ ($|H|\geq d+1$) which covers $\Re^d$. Then there is a subset $H' \subset H$ such that $|H'|=d+1$ and $H'$ covers $\Re^d$.
\end{proposition}

\begin {proof}
Assume to the contrary that any subset of $H$ which covers $\Re^d$ is of size at least $d+2$. Let $H'=\{h_1,...,h_n\}\subset H$ $(n\geq d+2)$ be a (containment) minimal subset which covers $\Re^d$, i.e., a subset of half-spaces such that for any half-space $h\in H'$, the set $H'\setminus \{h\}$ does not cover $\Re^d$. Since $H'$ is minimal, there exists a point set $P=\{p_1,...,p_n\}$ such that every $p_i$ is contained in a half-space $h_i\in H'$ and is not contained in any other half-space in $H'$, i.e., $p_i\in h_i\setminus \bigcup _{j\neq i}h_j$,  for every $1\leq i\leq n$. It easy to see that the set $P$ is in convex position. Since $|P|\geq d+2$ by Radon's lemma (see, e.g., \cite{matousek}), there exists a partition $P=P_1\cup P_2$ such that $P_1\cap P_2=\emptyset$ and $CH(P_1)\cap CH(P_2)\neq \emptyset$. Let $r\in CH(P_1)\cap CH(P_2)$. Notice that, since $P$ is in convex position, $r\notin P$. Also note that for every half-space $h$ containing $r$ we have $h\cap P_1\neq \emptyset$ and $h\cap P_2\neq \emptyset$, and since $r\notin P$, $|h\cap P|\geq 2$. We assumed that $H'$ covers $\Re^2$. Let $h_i\in H'$ be a half-space that contains $r$. We have  $|h_i\cap P|\geq 2$, a contradiction to the fact that $h_i\cap P=\{p_i\}$.
\end {proof}

As mentioned, we need the notion of point-line polar duality. A point $p=(a,b)$ (such that $(a,b)\neq (0,0)$) is transformed to a line $p^*$ with parameterization $ax+by=1$, and a line $l:ax+by=1$ is transformed to a point $l^*=(a,b)$. For the proof of Theorem \ref{dual-better}, we need the following proposition regarding point-line polar duality.

\begin{proposition}
\label{fact-polar-duality}
Let $H$ be a set of half-planes such that $H$ does not cover $\Re^2$. Let $q\in \Re^2$ be a point not covered by $H$. Assume without loss of generality that $q$ is the origin, $q=(0,0)$. Let $p$ be a point contained in a half-plane $h\in H$, and let $l_h$ be the boundary line of $h$. Then the dual line $p^*$ intersects the segment $\overline{ql_h^*}$, where $l_h^*$ is the point dual to $l_h$. See Figure \ref{fig-fact-duality} for an illustration. 
\end{proposition}

\begin{figure}

\centering

\subfigure[primal space]{
\includegraphics[scale=1]{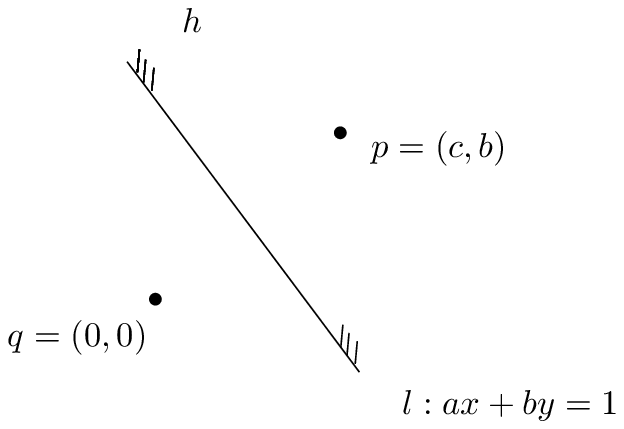}
}

\subfigure[dual space]{
\includegraphics[scale=1]{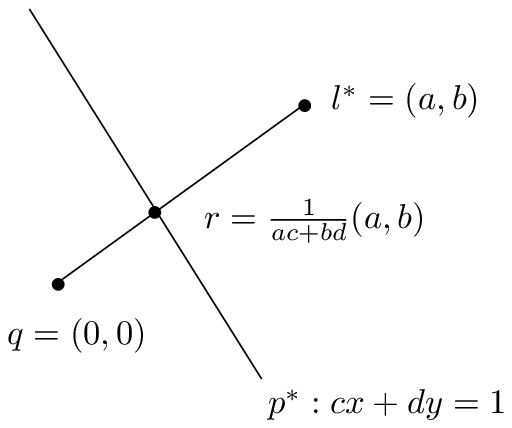}
}

\caption{\label{fig-fact-duality} point-line polar duality example}
\end{figure}

\begin{proof}
Assume that $l_h$ has parameterization $ax+by=1$ and $p=(c,d)$. We know that $h$ is a half-plane that does not contain the origin, i.e. $h=\{(x,y)|ax+by\geq 1\}$, therefore $ac+bd\geq 1$. The line that passes through the points $q=(0,0)$ and $l_h^*=(a,b)$ and the line $p^*$ intersect in the point $r=(\frac{a}{ac+bd},\frac{b}{ac+bd})=\frac{1}{ac+bd}(a,b)$. Since $ac+bd\geq 1$, $r$ must lie on the segment $\overline{ql_h^*}$. Hence the line $p^*$ intersects the segment $\overline{ql_h^*}$ as claimed.
\end{proof}

We are ready to prove Theorem \ref{dual-better}.

\begin {proof}
Let $H$ be a finite set of half-planes. If $H$ covers $\Re^2$, by Proposition \ref{fact-halfplanes}, we know that there is a subset $H_1 \subset H$ of 3 half-planes that covers $\Re^2$. If $H\setminus H_1$ covers $\Re ^2$, again by Proposition \ref{fact-halfplanes} there is a subset $H_2 \subset H\setminus H_1$ of 3 half-planes that covers $\Re^2$. Let $i$ be the number of subsets of half-planes that we find in such a way until we obtain a subset $H'=H\setminus \bigcup _{j=1}^{i} H_j$ such that $H'$ does not cover $\Re^2$. 

For every $1\leq j \leq i$, we assign to the 3 half-planes in $H_j$ the color $j$. If $i<k$ we color the half-planes in $H'$ with $k-i$ colors $\{i+1,i+2,...,k\}$ in the following way. Since $H'$ does not cover $\Re ^2$, there is a point $q\in \Re^2$ that is not covered by any half-plane in $H'$. Without loss of generality, assume that $q$ is the origin. Let $L_{H'}$ be the set of the bounding lines of the half-planes in $H'$. We transform the set of lines $L_{H'}$ to a set of points $L_{H'}^*$ using point-line polar duality. We color the points in $L_{H'}^*$ with $k-i$ colors using the colors $\{i+1,i+2,...,k\}$ such that every half-plane that contains at least $2(k-i)-1$ points is polychromatic. Such a coloring exists by Theorem \ref{thm-main}. 

Next, we show that the above coloring ensures that every point covered by at least $3k-2$ half-planes is also covered by some half-plane from every color class.
If $i\geq k$ then, by our construction every point in the plane is covered by a half-plane from every color class.
If $i<k$, let $p\in \Re^2$ be some point that is covered by at least $3k-2$ half-planes, $p$ is covered by at most $3i$ half-planes from the set $\bigcup _{j=1}^{i} H_j$, therefore $p$ is covered by at least $3k-2-3i=3(k-i)-2$ half-planes from $H'$. Since $k-i\geq 1$ we have that $p$ is covered by at least $2(k-i)-1$ half-planes from $H'$. Denote those covering half-planes by $H''$. By the property of our coloring of $H'$ we have that $H''$ contains a half-plane from every color class $c\in \{i+1,i+2,...,k\}$. Hence the point $p$ is polychromatic, i.e., covered by a half-plane from every color class $\{1,...,k\}$. This completes the proof.


\end {proof}

{\bf \noindent Remark:} Obviously, the result of Fulek
\cite{Fulek} implies that the bound $3k-2$ is not tight already
for $k=2$. It would be interesting to find the exact value of
$\bar{f}_{\H}(k)$ for every integer $k$.


\section{Small Epsilon-Nets for Half-Planes}

Consider a hypergraph $H=(P,E)$ where $P$ is a set of $n$ points
in the plane and $E=\{P\cap h: h \in \H\}$. As mentioned in the
introduction, Woeginger \cite{GW} showed that for any $1\geq\epsilon >
0$ there exists an $\epsilon$-net for $H$ of size at most
$\frac{2}{\epsilon}-1$.

As a corollary of Theorem~\ref{thm-main}, we obtain yet another
proof for this fact. Recall that for any integer $k \geq 1$ we
have $f_{\cal H}(k) \leq 2k-1$. Let $\epsilon > 0 $ be a fixed
real number. Put $k=\lceil\frac{\epsilon n+1}{2}\rceil$. Let
$\chi$ be a $k$-coloring as in Theorem~\ref{thm-main}. Notice that
every half-plane containing at least $\epsilon n$ points contains
at least $2k-1$ points of $P$. Indeed such a half-pane must
contain at least $\lceil \epsilon n\rceil = \lceil
2(\frac{\epsilon n+1}{2})-1\rceil = 2k-1$. Such a half-plane is
polychromatic with respect to $\chi$. Thus, every color class of
$\chi$ is an $\epsilon$-net for $H$. Moreover, by the pigeon-hole
principle one of the color classes has size at most $\frac{n}{k}
\leq \frac{2n}{\epsilon n+1} < \frac{2}{\epsilon}$. Thus such a
set contains at most $\frac{2}{\epsilon}-1$ points as asserted.

The arguments above are general and, in fact, we have the
following theorem:

\begin{theorem}
Let $\cal R$ be a family of regions such that $f_{\cal R}(k) \leq
ck$ for some absolute constant $c$ and every integer $k$. Then
for any $\epsilon$ and any finite set $P$ there exists an
$\epsilon$-net for $P$ with respect to $\cal R$ of size at most
$\frac{c}{\epsilon} -1$.
\end{theorem}

Applying the above theorem for the dual range space defined by a
set of $n$ half-planes with respect to points and plugging
Theorem~\ref{dual-better} we conclude that there exists an
$\epsilon$-net for such a range-space of size at most
$\frac{3}{\epsilon}-1$. However, using a more clever analysis one
can, in fact, show that for every $\eps\leq \frac{2}{3}$ there is an $\epsilon$-net of size at most $\frac{2}{\epsilon}$ for such a range-space. Indeed, let $H$ be a set of half-planes. If $H$ does not cover $\Re^2$, by using the polar point-line transformation we can obtain a coloring of the half-planes $H$ such that every point that is covered by at least $2k-1$ half-planes is polychromatic. Hence there is an $\eps$-net of size $\frac{2}{\epsilon}-1$. If $H$ covers $\Re^2$ then by Proposition \ref{fact-halfplanes} we can find a set of 3 half-planes $G\subset H$ such that $G$ covers $\Re^2$. The set $G$ is an $\eps$-net, i.e. for $\eps\leq \frac{2}{3}$ we have an $\eps$-net of size $\frac{2}{\eps}$.

\subsubsection*{Acknowledgments.} We wish to thank Panagiotis Cheilaris and Ilan Karpas for helpful discussions concerning the problem studied in this paper.


\bibliographystyle{plain}
\bibliography{reference}

\end{document}